\documentclass{article}
\usepackage{amssymb,amsmath,amsthm}
\usepackage{graphicx}
\usepackage[T1]{fontenc}
\usepackage{yfonts}

\newtheorem{theorem}{Theorem}

\newtheorem{example}{Example}

\def\bE{{\mathbb E}}

\def\bR{{\mathbb R}}
\def\Eq{{\bE}^q}
\def\Rq{{\bR}^q}

\def\bT{{\mathbb T}}

\def\cF{{\mathcal F}}

\def\cH{{\mathcal H}}

\def\Rt{{\bR^2}}

\def\span{\mathop{span}}

\def\Imm{\mathop{Imm}}
\def\Free{\mathop{Free}}

\title{Partial immersions and partially free maps}

\author{Roberto De Leo
\thanks{Dipartimento di Matematica e Informatica, Universit\`a di Cagliari, 
Via Ospedale 72, 09124 Cagliari, Italy $\langle$deleo@unica.it$\rangle$.}
\thanks{INFN, Sezione di Cagliari, Complesso Universitario di Monserrato, 
S.P. per Sestu Km 0.700, 09042 Monserrato (Cagliari),
Italy $\langle$roberto.deleo@ca.infn.it$\rangle$.}}

\begin{document}
\maketitle

{\sl Math Subject Classification:} Primary: 58A30, 58J60, 53C12, 53C20, 37J35, 53D17

{\sl Keywords:} Partially isometric immersions, Partially Free maps, Distributions, 
Completely Integrable Systems, Poisson Manifolds

\begin{abstract}
  In a recent paper~\cite{DDL10} we studied basic properties of partial immersions and
  partially free maps, a generalization of free maps introduced first by Gromov in~\cite{Gro70}.
  In this short note we show how to build partially free maps out of partial immersions
  and use this fact to prove that the partially free maps in critical dimension 
  introduced in Theorems 1.1-1.3 of~\cite{DDL10} for three important types of distributions
  can actually be built out of partial immersions.
  Finally, we show that the canonical contact structure on $\bR^{2n+1}$ admits 
  partial immersions in critical dimension for every $n$.
\end{abstract}
\section{Introduction}
In a joint paper with D'Ambra and Loi~\cite{DDL10}, in analogy with the 
theory of $C^k$, $k\geq3$, isometric immersions of Nash and Gromov~\cite{Nas56,Gro86},
we studied basic properties of {\em partial} isometric immersions 
(also called {\em $\cH$-immersions}), namely $C^1$ maps of a manifold $M$ into the
Euclidean space $\Rq$
which induce a metric on some vector subbundle $\cH$ of $M$.
In particular we proved (see Theorems 1.1-1.3 in~\cite{DDL10}), 
by an explicit construction, the existence of $\cH$-free 
maps, the analog of free maps for partial isometric immersions, in critical dimension 
for three types of distributions: 1-dimensional planar distributions which are either 
Hamiltonian or of finite type; $n$-dimensional Lagrangian distributions of completely 
integrable systems in a $2n$-dimensional symplectic manifold; 1-dimensional Hamiltonian 
distributions in a Riemann-Poisson manifold.

In this note we show that $\cH$-free maps can be canonically built out of a 
$\cH$-immersion. Accordingly, we show that Theorems 1.1-1,3 of~\cite{DDL10} ultimately
depend on the fact that those distributions admit a $\cH$-immersion in critical dimension.
Moreover, we add to the list the canonical contact distributions in $\bR^{2n+1}$.
\section{$\cH$-immersions and $\cH$-free maps}
Let $M$ be a smooth $m$-dimensional manifold, $\cH$ a $k$-dimensional distribution 
on $M$, namely a vector subbundle of $TM$ such that $\dim\cH_x=k$ for all $x\in M$, 
and $\bE^q$ the $q$-dimensional Euclidean space, namely the linear space $\Rq$ endowed 
with the euclidean metric $e_q=\delta_{ij}dy^idy^j$, where $(y^i)$, $i=1,\dots,q$, are 
linear coordinates on $\Rq$.
Recall that, locally, $\cH$ is the span of $k$ vector fields $\{\xi_1,\dots,\xi_k\}$;
the $\{\xi_a\}_{a=1,\dots,k}$, are called a {\em local trivialization} of $\cH$.
We say that a $C^1$ map $f=(f^i):M\to\Eq$ is a $\cH$-immersion when the $k$
vectors $\{L_{\xi_a}f^i\,\partial_i\}_{a=1,\dots,k}$ are linearly independent 
(equivalently, when the $k\times q$ matrix 
$$
D_1(f)=\left(L_{\xi_a}f^i\right)
$$ 
is full-rank) at every point 
and for every local trivialization. Similarly, we say that a $C^2$ map $f=(f^i):M\to\Eq$ is 
$\cH$-free if the $k+s_k$ vectors $\{L_{\xi_a}f^i\,\partial_i,\{L_{\xi_a},L_{\xi_b}\}f^i\,\partial_i)$ 
are linearly independent (equivalently, when the $(k+s_k)\times q$ matrix 
$$
D_2(f)=\begin{pmatrix}L_{\xi_a}f^i\cr\{L_{\xi_a},L_{\xi_b}\}f^i\cr\end{pmatrix}
$$ 
is full-rank) at every point of every local trivialization, 
where $s_k=k(k+1)/2$ and $\{L_{\xi_a},L_{\xi_b}\}=L_{\xi_a}L_{\xi_b}+L_{\xi_b}L_{\xi_a}$ is 
the anticommutator of $L_{\xi_a}$ and $L_{\xi_b}$ (the definition of $\cH$-free map
was first introduced by Gromov in~\cite{Gro70}). Clearly $TM$-immersions
are the usual immersions and $TM$-free maps are the usual free maps.

We denote by $\Imm_\cH(M,\Rq)$ and $\Free_\cH(M,\Rq)$ the sets of $\cH$-immersions
and $\cH$-free maps $M\to\bE^q$ and endow $C^\infty(M,\Rq)$ with the strong 
Whitney topology. Both $\Imm_\cH(M,\Rq)$ and $\Free_\cH(M,\Rq)$ are open
subsets of $C^\infty(M,\Rq)$ and are clearly empty for, respectively,
$q<k$ and $q<k+s_k$ (we say that $k$ and $k+s_k$ are
{\em critical dimensions} for, respectively, $\cH$-immersions and
$\cH$-free maps). Next theorem shows that, independently on the
topology of $\cH$ and of $M$, both sets are non-empty if $q$ is big enough:
\begin{theorem}[DDL, 2010]
  The sets $\Imm_\cH(M,\Rq)$ and $\Free_\cH(M,\Rq)$ are dense
  in $C^\infty(M,\Rq)$ for, respectively, $q\geq m+k$ and $q\geq m+k+s_k$.
\end{theorem}
What happens in general in the range $k\leq q<m+k$ for $\cH$-immersions 
and $k+s_k\leq q<m+k+s_k$ for $\cH$-free maps is still an open question. 
When $\cH=TM$ it is known that the h-principle holds for immersions 
(resp. free maps) for 
$q\geq n$ (resp. $q\geq m+s_m$) when $M$ is open and for $q>m$ (resp. 
$q>m+s_m$) when $M$ contains a closed component (see~\cite{EG71} 
and~\cite{Gro86}). This means that, under those conditions, free maps
arise whenever the whenever the appropriate topological obstructions vanish.
\begin{example}
  The set $\Free(\bR^m,\bR^{m+s_m})$ is non empty. A concrete element of
  that set is the polynomial map
  $$
  F_m(x^1,\dots,x^m) = (x^1,\dots,x^m,(x^1)^2,x^1x^2,\dots,(x^m)^2)
  $$
  of all possible monic monomials of first and second degree in the 
  coordinates.
\end{example}
%
The critical dimension case for immersions is trivial since no compact 
$m$-manifold can be immersed into $\bR^m$. The question of the existence 
of free maps in critical dimension on compact sets is instead of particular 
interest and still open. For example, it is still unknown whether the tori
$\bT^m$, $m>1$, admit free maps in critical dimension (see~\cite{Gro86}, Section 1.1.4). 

When $\cH\neq TM$ 
the theory is richer since $\cH$-immersions can arise also in
critical dimension, as next example shows:
\begin{example}
  \label{ex:Free}
  Let $\xi$ be a vector field without zeros on a Riemannian manifold
  $(M,g)$ and $\cH=\span\{\xi\}$. Assume that the 1-form $\xi^\flat$, 
  obtained by ``raising the index'' of $\xi$, is {\em intrinsically exact}, namely that
  $\xi^\flat=\lambda df$ for some smooth functions $f$ and $\lambda>0$. 
  Then 
  $L_\xi f= \|\xi\|_g^2/\lambda>0$, so that $f\in\Imm_\cH(M,\bR)$. 
  For example consider 
  $$
  \xi(x,y)=y(1-y^2)\partial_x+(1-3y^2)\partial_y
  $$ 
  in $\Rt$. Then $\xi^\flat=e^{-x}d(y(1-y^2)e^x)$ and therefore 
  $$
  L_\xi(y(1-y^2)e^x) = y^2(1-y^2)^2 + (1-3y^2)^2 > 0\,,
  $$
  namely $y(1-y^2)e^x\in\Imm_{\cH}(\Rt,\bR)$. Observe that $\xi$
  is not topologically conjugate to a constant vector field so
  that, in principle, the solvability of $L_\xi f>0$ is not 
  a trivial matter.
\end{example}
Moreover, next theorem shows that $\cH$-immersions can be used to
build $\cH$-free maps:
%
\begin{theorem}
  \label{thm:main}
  Let $F\in \Free(\Rq,\bR^{q'})$ and $f\in\Imm_\cH(M,\Rq)$. 
  Then $F\circ f\in\Free_\cH(M,\bR^{q'})$.
  In particular, if $\cH$
  admits $\cH$-immersions in critical dimension, then $\cH$ admits $\cH$-free
  maps in critical dimension.
\end{theorem}
\begin{proof}
  We can prove the claim without loss of generality in the critical dimension
  case, namely when $q=k$ and $q'=k+s_k$.
  Let $f=(f^a):M\to\bR^k$ a $\cH$-immersion and $F=(F^i):\bR^k\to\bR^{k+s_k}$
  a free map (see Example~\ref{ex:Free}). 

  Now observe that
  \begin{equation*}
    \begin{cases}
      L_{\xi_a}F^i(f^1,\dots,f^k) &= L_{\xi_a}f^c \, \partial_c F\cr
      L^2_{\xi_a}F^i(f^1,\dots,f^k) &= L_{\xi_a}f^c \, L_{\xi_a}f^d \, \partial^2_{cd} F + L^2_{\xi_a}f^c \, \partial_c F\cr
      \{L_{\xi_a},L_{\xi_b}\}F^i(f^1,\dots,f^k) &= 2 L_{\xi_a}f^c \, L_{\xi_b}f^d \, \partial^2_{cd} F + \{L_{\xi_a},L_{\xi_b}\}f^c \, \partial_c F\cr
    \end{cases}
  \end{equation*}
  This shows that 
  $$
  D_2(F\circ f) =
  \left(
  \begin{array}{c|c}
    D_1(f)&\mathbb O_{k,s_k}\cr
    \hline
    C&D\cr
  \end{array}
  \right)
  D_2(F)
  $$
  where 
  $$
  C=
  \left(
  \begin{array}{ccc}
    L^2_{\xi_1}f^1&\dots&L^2_{\xi_1}f^k\cr
    \{L_{\xi_1},L_{\xi_2}\}f^1&\dots&\{L_{\xi_1},L_{\xi_2}\}f^k\cr
    \vdots&\vdots&\vdots\cr
    L_{\xi_k}f^1&\dots&L_{\xi_k}f^k\cr
  \end{array}
  \right)
  $$
  and
  $$
  D = 
  \left(
  \begin{array}{cccc}
    (L_{\xi_1}f^1)^2&2L_{\xi_1}f^1 L_{\xi_1}f^2&\dots&(L_{\xi_1}f^k)^2\cr
    L_{\xi_1}f^1 L_{\xi_2}f^1&L_{\xi_1}f^1 L_{\xi_2}f^1+L_{\xi_2}f^1 L_{\xi_1}f^1&\dots&L_{\xi_1}f^k L_{\xi_2}f^k\cr
    \vdots&\vdots&\vdots&\vdots\cr
    (L_{\xi_k}f^1)^2&2L_{\xi_1}f^1 L_{\xi_1}f^2&\dots&(L_{\xi_1}f^k)^2\cr
  \end{array}
  \right)\,.
  $$
  Clearly then $\det D_2(F\circ f)=\det D_1(f)\det D\det D_2 F$. It is easy to check that
  the matrix $D$ can be written as $\rho(D_1(f))$, where 
  $\rho:GL_{k}(\bR)\to GL_{s_k}(\bR)$ is a linear representation of $GL_{k}(\bR)$ over $\bR^{s_k}$,
  and therefore $\det D_1(f)\neq0$ implies $\det D\neq0$. In particular it is easy to check
  (e.g. it is enough to consider the case of diagonal matrices) that $\det D=(\det D_1(f))^{k+1}$,
  so that 
  $$
  \det D_2(F\circ f)=(\det D_1(f))^{k+2}\det D_2 F\,.
  $$
  Hence if $f$ is a $\cH$-immersion and $F$ is free, the map $F\circ f$ is $\cH$-free.
\end{proof}
\section{$\cH$-immersions in critical dimension}
Thanks to Theorem~\ref{thm:main}, we can now reformulate Theorems 1.1-1.3 of~\cite{DDL10} 
so that it is clear that they all depend on the existence of a $\cH$-immersion.

Consider first the case of 1-distributions $\cH$ in the plane (see Section 3.1 in~\cite{DDL10}). 
Kaplan proved~\cite{Kap40} that 
all 1-distributions in the plane are orientable, so that there exists a vector field everywhere
non-zero such that $\cH=\span \xi$. We say that $\cH$ is {\em Hamiltonian} when it is tangent
to the level sets of a regular\footnote{We say that a smooth function is {\em regular} when $df\neq0$
at every point. Analogously, we say that a vector field is regular when it has no zeros.} 
function $f$, i.e. $\cH=\ker df$. 
Let now $\cF$ be the integral foliation of $\cH$. Two leaves are said {\em separatrices} when
they cannot be separated in the quotient topology on $\cF$. We say that $\cH$ is of 
{\em finite type} when the set of the separatrices of $\cF$ is closed and every separatrix is
inseparable from just a finite number of other leaves; for example, if $\cH$ is the span
of a polynomial vector field then it is of finite type~\cite{Mar72}.
\begin{theorem}
  Let $\cH$ be a planar 1-distribution which is either Hamiltonian or of finite type.
  Then $\cH$ admits a $\cH$-immersion in critical dimension. 
\end{theorem}
\begin{proof}
  A $\cH$-immersion $f:\Rt\to\bR$ is a function $f$ such that either $L_\xi f>0$ or $L_\xi f<0$
  at every point. When $\cH$ is Hamiltonian, the existence of such a function was
  proved by Weiner in its Lemma in~\cite{Wei88}. When $\cH$ is of finite type, it
  was proved in Lemma 3.1 of~\cite{DDL10}.
\end{proof}
\begin{example}
  \label{ex:iH}
  Consider the distribution $\cH_\xi=\span\xi\subset T\Rt$, with 
  $\xi=2y\partial_x + (1-y^2)\partial_y$. Since the components of $\xi$
  depend only on $y$, only vertical straight lines can be separatrices
  for $\cH_\xi$; in particular only $y=\pm1$ are separatrix leaves for $\cH_\xi$. 
  A direct calculation shows that $\ker L_\xi$ is functionally generated by the regular 
  smooth function $f(x,y)=(1-y^2)e^x$, namely $\cH_\xi=\ker df$ is a Hamiltonian
  distribution. Now let $g(x,y)=ye^x$.
  It is easy to check that $L_\xi g(x,y) = (1+y^2)e^x>0$, so that 
  $g\in\Imm_{\cH_\xi}(\Rt,\bR)$ and $F=(g,g^2)\in\Free_{\cH_\xi}(\Rt,\Rt)$.

  Consider now the distribution $\cH_\eta=\span\eta$, with 
  $\eta=(3y-1)\partial_x+ (1-y^2)\partial_y$. Considerations similar to the
  ones made above show that $y=\pm1$ are the only separatrices 
  for $\cH_\eta$. A functional generator for $\ker L_\eta$ is given by 
  $f'(x,y)=(1-y)(1+y)^2e^x$, whose gradient is null on the separatrix $y=-1$,
  so that $\cH_\eta$ is not Hamiltonian; nevertheless $\cH_\eta$ is of finite type 
  since $\eta$ is polynomial. A direct calculation shows that
  $L_\eta g(x,y)=(2y^2-y+1)e^x>0$, so that $g\in\Imm_{\cH_\eta}(\Rt,\bR)$ 
  and $F=(g,g^2)\in\Free_{\cH_\eta}(\Rt,\Rt)$.
\end{example}
Consider now the case of completely integral systems (Section 3.2 in~\cite{DDL10}). 
Recall that, on a $2n$-dimensional symplectic manifold $(M,\omega)$, a completely integrable 
system is a collection of $n$ functions $\{h_1,\dots,h_n\}$ in involution, i.e. such that
the Poisson bracket of every pair of $h_i$ is identically zero. 
%
\begin{theorem}
  Let $\cH=\cap_{i=1}^n\ker dI_i$ be the Lagrangian $n$-distribution of a completely integrable system 
  $\{h_1,\dots,h_n\}$ on a symplectic manifold $(M^{2n},\omega)$ such that:
  \begin{enumerate}
    \item the Hamiltonian vector fields $\xi_i$ of the $h_i$ are all complete;
    \item the (Lagrangian) leaves of $\cH$ are all diffeomorphic to $\bR^n$.
  \end{enumerate}
  Then $\cH$ admits a $\cH$-immersion $f:M^{2n}\to\bR^n$.
\end{theorem}
\begin{proof}
  In Lemma 3.3 of~\cite{DDL10} we proved the existence, under the same assumptions
  of this theorem, of $n$ functions $f^i$ such that $\{h_i,f^j\}=0$ for $i\neq j$ 
  and $\{h_i,f^i\}>0$, $i=1,\dots,n$. Since $\cH$ is spanned by the Hamiltonian 
  pairwise commuting vector fields $\xi_i$ and $L_{\xi_i}f^j=\{h_i,f^j\}$, this is enough to grant
  that the map $f=(f^1,\dots,f^n):M\to\bR^n$ is a $\cH$-immersion.
\end{proof}
\begin{example}
  Consider the symplectic manifold $T^*\bT^n$ with canonical coordinates 
  $(\varphi^\alpha,p_\alpha)$, so that the symplectic form is equal to 
  $\omega=d\varphi^\alpha\wedge dp_\alpha$.
  The system $\{I_\alpha=e^{p_\alpha}\cos\varphi^\alpha\}_{\alpha=1,\dots,n}$ is  
  completely integrable on $T^*\bT^n$, e.g. because each $I_\alpha$ depends
  only on the two coordinates with index $\alpha$.
  
  The corresponding Lagrangian distribution $\cH=\cap_{\alpha=1}^n\ker dI_\alpha$ 
  is generated by the pairwise commuting (Hamiltonian) vector fields 
  $\xi_\alpha=e^{p_\alpha}(\sin\varphi^\alpha,\cos\varphi^\alpha)$.
  This system is clearly the direct product of $n$ independent systems on
  the cilinders $(\varphi^\alpha,p_\alpha)$, $\alpha=1,\dots,n$, in such a way 
  that the $\alpha$-th system admits, as partial immersion, the function
  $g_\alpha=e^{p_\alpha}\sin\varphi^\alpha$; indeed $L_{\xi_\alpha}g_\alpha=e^{2p_\alpha}>0$. 
  Hence the map 
  $$
  G=(g_1,\dots,g_n):T^*\bT^n\to\bR^n
  $$ 
  is a $\cH$-immersion and, consequently, the map 
  $$
  F=(g_1,\dots,g_n,g_1^2,g_1g_2,\dots,g_n^2):T^*\bT^n\to\bR^{n+s_n}
  $$ 
  is a $\cH$-free map.
\end{example}
Finally, consider the case of Riemann-Poisson manifolds. These are Riemannian
manifolds $(M,g)$ on which it is defined the Poisson structure
$$
\{f,g\}_H\stackrel{\rm def}{=}*[dh_1\wedge\cdots\wedge dh_{m-2}\wedge df\wedge dg]\,,
$$
where the $H=\{h_1,\cdots,h_{m-2}\}$ are fixed smooth (possibly multivalued) functions 
on $M$.
\begin{example}
  Consider the flat torus $\bT^3$ with angular coordinates $(\theta^1,\theta^2,\theta^3)$ and 
  $H=\{h(\theta^i)=B_i\theta^i\}$ for some constant 1-form $B=B_i d\theta^i$.
  Then the Riemann-Poisson bracket is given by 
  $$
  \{f,g\}_H = \epsilon^{ijk}\partial_i f\,\partial_i g\, B_k\;,
  $$
  where $\epsilon^{ijk}$ is the totally antisymmetric Levi--Civita tensor. 
  This bracket was introduced by S.P.~Novikov as an application of his 
  generalization of Morse theory to multivalued functions~\cite{Nov82}. 
  An example of the rich topological structure hidden in this Riemann-Poisson 
  bracket can be found in~\cite{DD09}.
\end{example}
\begin{theorem}
  Let $(M,g,\{,\}_H)$ a Riemann-Poisson manifold such that the $m-2$ functions 
  in $H$ are functionally independent at every point and let $\cH$ be a 
  Hamiltonian $1$-distribution on it. 
  Then $\cH$ admits a $\cH$-immersion $f:M\to\bR$.
\end{theorem}
\begin{proof}
  Let $h$ a Hamiltonian for $\cH$.
  A $\cH$-immersion $f:M\to\bR$ is a function $f$ such that $\{h,f\}_H>0$ 
  (or $\{h,f\}_H<0$). The existence of such a function was proven in Lemma
  3.4 of~\cite{DDL10}.
\end{proof}
\begin{example}
  Consider the case of $\bE^3$ with the Riemann-Poisson structure induced
  by the singlet $H=\{(1-y^2)e^x\}$, so that
  $$
  \{f,g\}_H = e^x\left[
  (1-y^2)\left(\partial_y f\partial_z g-\partial_z f\partial_y g\right)
  -2y\left(\partial_x f\partial_z g-\partial_z f\partial_x g\right)
  \right]\,.
  $$
  Take a Hamiltonian of the form $h(x,y,z)=\lambda(x,y)z+\mu(x,y)$,
  where $\lambda$ is strictly positive and $\mu$ is arbitrary.
  The Hamiltonian 1-dimensional distribution $\cH$ corresponding 
  to $h$ is the span of the regular vector field $\xi_h=\{h,\cdot\}_H$
  which, for our particular choice of $h$, writes as
  $$
  \xi_h = e^x\left[
  (1-2y-y^2)(z\partial_y \lambda(x,y) +\partial_y \mu(x,y))\partial_z - 
    \lambda(x,y)\left((1-y^2)\partial_y-2y \partial_x\right)
  \right].
  $$
  Example~\ref{ex:iH} shows that $f(x,y,z)=ye^x$ solves the partial differential 
  inequality $\{h,f\}_H>0$. Indeed
  $$
  \{h,f\}_H = L_{\xi_h} f = e^x\lambda(x,y)\left[(1-y^2) e^x + 2y^2e^x\right] = (1+y^2)\lambda(x,y) e^{2x}>0\,,
  $$
  so that $f\in\Imm_\cH(\bR^3,\bR)$ and $(f,f^2)\in\Free_\cH(\bR^3,\bR)$.
\end{example}
We add now a fourth case where it is possible to find a $\cH$-immersion
in critical dimension. Recall that a contact structure on a 
$(2n+1)$-dimensional manifold $M$ is a completely non-integrable codimension-1
distribution $\cH\subset TM$. Locally $\cH=\ker\theta$ for some 1-form
$\theta$, so that the non-integrability condition translates into 
$\theta\wedge(d\theta)^n\neq0$. 
\begin{example} 
  Consider the bundle $J^1(N,\bR)\simeq T^*N\times\bR$ of all 1-jets of maps $N\to\bR$,
  where $N$ is a $n$-dimensional manifold.
  This bundle has a canonical contact structure induced 
  by the tautological 1-form $\theta$, defined as the unique (modulo strictly positive 
  or negative smooth functions) 1-form
  such that a section $\sigma:N\to J^1(N,\bR)$ is holonomic (namely is the 1-jet of a 
  map $M\to\bR$) iff $\sigma^*\theta=0$.
  In canonical coordinates 
  $(x^\alpha,p_\alpha,t)$ a canonical contact form writes as
  $
  \theta = \lambda(x,p,t)\left(dt - p_\alpha dx^\alpha\right),
  $
  where $\lambda(x,p,t)$ is never zero.
  For $N=\bR^n$ and this gives exactly the canonical contact structure on 
  $\bR^{2n+1}\simeq J^1(\bR^n,\bR)$.
\end{example}
\begin{theorem}
  Let $\cH$ be the canonical contact structure on $\bR^{2n+1}$.
  Then $\cH$ admits a $\cH$-immersion in critical dimension.
\end{theorem}
\begin{proof}
  A trivialization for $\cH$ is given by the $2n$ vectors
  $$
  \xi_{1}=\partial_{x^1} - p_1\partial_t\,,\;\dots\,,\;\xi_{n}=\partial_{x^n} - p_n\partial_t\,,
  \xi_{n+1}=\partial_{p_1}\,,\;\dots\,,\;\xi_{2n}=\partial_{p_n}\,.
  $$
  Hence the projection on the first $2n$ components 
  $$
  \pi(x^1,p_1,\dots,x^n,p_n,z) = (x^1,p_1,\dots,x^n,p_n)
  $$ 
  belongs to $\Imm_\cH(\bR^{2n+1},\bR^{2n})$ and $F_{2n}\circ \pi$ belongs to
  $\Free_\cH(\bR^{2n+1},\bR^{2n+s_{2n}})$.
\end{proof}
\bibliography{refs}
\end{document}